\documentclass[reqno,12pt]{amsart}

\usepackage{amsmath}
\usepackage{amsfonts}
\usepackage{amssymb}
\usepackage{eufrak}
\usepackage{amscd}
\usepackage{amsthm}
\usepackage{amstext}
\usepackage[all]{xy}
   \newcommand{\Sp}{\operatorname{Sp}}

\newcommand{\Ker}{\operatorname{Ker}}
\newcommand{\id}{\operatorname{id}}

\renewcommand{\Im}{\operatorname{Im}}
\newcommand{\ev}{\operatorname{ev}}

\newcommand{\tr}{\operatorname{tr}}

   \theoremstyle{plain}%default
   \newtheorem{thm}{Theorem}[section]
   
   \newtheorem{lem}[thm]{Lemma}
   \newtheorem{cor}[thm]{Corollary}
   \theoremstyle{definition}

   \newtheorem{example}[thm]{Example}

   \newtheorem{remark}[thm]{Remark}

\title{Weakening idempotency in $K$-theory}

\author{V. Manuilov}

\date{}

\address{Moscow State University,
Leninskie Gory, Moscow, 
119991, Russia, and Harbin Institute of Technology,  Harbin, 
P. R. China}

\email{manuilov@mech.math.msu.su}

\thanks{The author acknowledges partial support by the RFBR grant No. 10-01-00257 and by the Russian Government grant No. 11.G34.31.0005.}

\begin{document}

\maketitle

\begin{abstract}
We show that the $K$-theory of $C^*$-algebras can be defined by pairs of matrices satisfying less strict relations than idempotency.

\end{abstract}

\section{Introduction}

$K$-theory of a $C^*$-algebra $A$ is patently defined by pairs (formal differences) of idempotent matrices (projections) over $A$. Regretfully, being a projection is a very strict property, and it is usually very hard to find projections in a given $C^*$-algebra. Many famous conjectures (Kadison, Novikov, Baum--Connes, Bass, etc.) are related to projections and would become more tractable if one could provide enough projections for a given $C^*$-algebra. Our aim is to show that the $K$-theory can be defined using less restrictive relations in hope that it would be easier to find elements satisfying these relations than the genuine idempotency. We show that $K$-theory is generated by pairs $a,b$ of matrices over $A$ satisfying $(a-a^2)(a-b)=(b-b^2)(a-b)=0$, which means that $a$ and $b$ have to be ``projections'' only when $a\neq b$.

\section{Definitions and some properties}

Let $A$ be a $C^*$-algebra. For $a,b\in A$, consider the relations 
\begin{equation}\label{main}
\|a\|\leq 1;\quad \|b\|\leq 1;\quad a,b\geq 0;\quad (a-a^2)(a-b)=0;\quad (b-b^2)(a-b)=0.
\end{equation}

Two pairs, $(a_0,b_0)$ and $(a_1,b_1)$ of elements in $A$, are {\it homotopy equivalent} if there are paths $a=(a_t),b=(b_t):[0,1]\to A$, connecting $a_0$ with $a_1$ and $b_0$ with $b_1$ respectively, such that the relations 
$$
\|a_t\|\leq 1;\quad \|b_t\|\leq 1;\quad a_t,b_t\geq 0;\quad (a_t-a_t^2)(a_t-b_t)=0;\quad (b_t-b_t^2)(a_t-b_t)=0
$$ 
hold for each $t\in[0,1]$.

A pair $(a,b)$ is {\it homotopy trivial} if it is homotopy equivalent to $(0,0)$.

\begin{lem}\label{L1}
The pair $(a,a)$ is homotopy trivial for any $a\in A$.

\end{lem}
\begin{proof}
The linear homotopy $a_t=t\cdot a$ would do.

\end{proof}

\begin{lem}\label{ideal1}
If $a,b$ satisfy (\ref{main}) then $f(a)=f(b)$ and $f(a)(a-b)=0$ for any $f\in C_0(0,1)$.

\end{lem}
\begin{proof}
It follows from $(a-a^2)(a-b)=0$, or, equivalently, from $(a-a^2)a=(a-a^2)b$, that 
$$
(a-a^2)a^2=a(a-a^2)a=a(a-a^2)b=(a-a^2)b^2, 
$$
hence 
$$
(a-a^2)(a-a^2)=(a-a^2)(b-b^2).
$$ 
Similarly, 
$$
(b-b^2)(b-b^2)=(a-a^2)(b-b^2),
$$ 
therefore 
\begin{equation}\label{kvadrat}
(a-a^2)^2=(b-b^2)^2. 
\end{equation}
Then (\ref{kvadrat}) and positivity of $a-a^2$ and $b-b^2$ imply that 
$$
a-a^2=b-b^2.
$$ 
Also, 
$$ 
(a-a^2)a=(a-a^2)b=(b-b^2)b.
$$
Since the two functions $g,h$, $g(t)=t-t^2$, $h(t)=tg(t)$, generate $C_0(0,1)$, and $g(a)=g(b)$, $h(a)=h(b)$, we conclude that the same holds for any $f\in C_0(0,1)$. Similarly, $g(a)(a-b)=0$ and $h(a)(a-b)=0$ implies $f(a)(a-b)=0$ for any $f\in C_0(0,1)$.  

\end{proof}

\begin{cor}
If $\|a\|<1$, $\|b\|<1$ and the pair $(a,b)$ satisfies (\ref{main}) then $a=b$, hence the pair $(a,b)$ is homotopy trivial.

\end{cor}
\begin{proof}
Take $f\in C_0(0,1)$ such that $f(t)=t\in\Sp(a)\cup\Sp(b)$ and $f(1)=0$. Then $a=f(a)$, $b=f(b)$, and the claim follows from Lemma \ref{ideal1}. 

\end{proof}

\begin{lem}
The pair $(f(a),f(b))$ is homotopy equivalent to $(a,b)$ for any continuous map $f:[0,1]\to[0,1]$ such that $f(0)=0$, $f(1)=1$.

\end{lem}
\begin{proof}
As the set of all functions with the stated properties is convex, so it suffices to show that for any such function $f$, the pair $(f(a),f(b))$ satisfies the relations (\ref{main}). 

Set $f_0(t)=f(t)-t$. Then $f_0\in C_0(0,1)$. As $f_0(a)=f_0(b)$ by Lemma \ref{ideal1}, so 
$$
f(a)-f(b)=a-b.
$$ 
Set 
$$
g(t)=t-t^2+f_0(t)-f_0^2(t)-2tf_0(t).
$$ 
Then $g\in C_0(0,1)$ and 
$$
(f(a)-f^2(a))(f(a)-f(b))=g(a)(a-b)=0; 
$$
$$
(f(b)-f^2(b))(f(a)-f(b))=g(a)(a-b)=0.
$$

\end{proof}

\begin{cor}
$\Sp(a)\setminus\{0,1\}=\Sp(b)\setminus\{0,1\}$.

\end{cor}
\begin{proof}
The inner points of $[0,1]$ in the two spectra must coinside by Lemma \ref{ideal1}. 
\end{proof}

Let $M_n(A)$ denote the $n{\times n}$ matrix algebra over $A$. Two pairs, $(a_0,b_0)$ and $(a_1,b_1)$, where $a_0,a_1,b_0,b_1\in M_n(A)$, are equivalent if there is a homotopy trivial pair $(a,b)$, $a,b\in M_m(A)$ for some integer $m$, such that the pairs $(a_0\oplus a,b_0\oplus b)$ and $(a_1\oplus a,b_1\oplus b)$ are homotopy equivalent in $M_{n+m}(A)$. Using the standard inclusion $M_n(A)\subset M_{n+k}(A)$ (as the upper left corner) we may speak about equivalence of pairs of different matrix size.

Let $[(a,b)]$ denote the equivalence class of the pair $(a,b)$, $a,b\in M_n(A)$.

For two pairs, $(a,b)$, $a,b\in M_n(A)$, and $(c,d)$, $c,d\in M_m(A)$, set 
$$
[(a,b)]+[(c,d)]=[(a\oplus c,b\oplus d)]. 
$$
The result obviously doesn't depend on a choice of representatives. Also $[(a,b)]+[(c,d)]=[(a,b)]$ when $(c,d)$ is homotopy trivial.

\begin{lem}
The addition is commutative and associative.

\end{lem}
\begin{proof}
If $(u_t)_{t\in[0,1]}$ is a path of unitaries in $A$, $u_1=1$, $u_0=u$, then $[(u^*au,u^*bu)]=[(a,b)]$ for any $a,b\in A$, as the relations (\ref{main}) are not affected by unitary equivalence. The standard argument with a unitary path connecting $\left(\begin{smallmatrix}1&0\\0&1\end{smallmatrix}\right)$ and $\left(\begin{smallmatrix}0&1\\1&0\end{smallmatrix}\right)$ proves commutativity. A similar argument proves associativity.

\end{proof}

\begin{lem}
$[(a,b)]+[(b,a)]=[(0,0)]$ for any $a,b$.

\end{lem}
\begin{proof}
Set $U_t=\left(\begin{smallmatrix}\cos t&-\sin t\\ \sin t&\cos t\end{smallmatrix}\right)$, $B_t=U_t^*\left(\begin{smallmatrix}b&0\\0&a\end{smallmatrix}\right)U_t$.
We claim that the pair $\bigl(\left(\begin{smallmatrix}a&0\\0&b\end{smallmatrix}\right),B_t\bigr)$ satisfies the relations (\ref{main}) for all $t$.

One has 
\begin{equation}\label{1}
B_t=\left(\begin{smallmatrix}b\cos^2t+a\sin^2t&(a-b)\cos t\sin t\\(a-b)\cos t\sin t&b\sin^2t+a\cos^2t\end{smallmatrix}\right)=
\left(\begin{smallmatrix}a&0\\0&b\end{smallmatrix}\right)+(a-b)C_t,
\end{equation}
where $C_t=\left(\begin{smallmatrix}-\cos^2t&\cos t\sin t\\ \cos t\sin t&\cos^2t\end{smallmatrix}\right)$.

Then 
$$
\left(\left(\begin{smallmatrix}a&0\\0&b\end{smallmatrix}\right)-\left(\begin{smallmatrix}a&0\\0&b\end{smallmatrix}\right)^2\right)\Bigl(\left(\begin{smallmatrix}a&0\\0&b\end{smallmatrix}\right)-B_t\Bigr)=\left(\begin{smallmatrix}a-a^2&0\\0&b-b^2\end{smallmatrix}\right)(a-b)C_t=\left(\begin{smallmatrix}(a-a^2)(a-b)&0\\0&(b-b^2)(a-b)\end{smallmatrix}\right)C_t=0.
$$

It remains to show that 
$$
A=(B_t-B_t^2)\bigl(\left(\begin{smallmatrix}a&0\\0&b\end{smallmatrix}\right)-B_t\bigr)=0.
$$
Using (\ref{1}) we have
\begin{eqnarray*}
A&=&
\bigl(\left(\begin{smallmatrix}a&0\\0&b\end{smallmatrix}\right)+(a-b)C_t-\bigl(\left(\begin{smallmatrix}a&0\\0&b\end{smallmatrix}\right)+(a-b)C_t\bigr)^2\bigr)
(a-b)C_t\\
&=&\bigl(\bigl(\begin{smallmatrix}a-a^2&0\\0&b-b^2\end{smallmatrix}\bigr)+(a-b)C_t-\left(\begin{smallmatrix}a&0\\0&b\end{smallmatrix}\right)(a-b)C_t-
C_t(a-b)\left(\begin{smallmatrix}a&0\\0&b\end{smallmatrix}\right)-(a-b)^2C_t^2\bigr)(a-b)C_t\\
&=&\bigl((a-b)C_t-\left(\begin{smallmatrix}a&0\\0&b\end{smallmatrix}\right)(a-b)C_t-
C_t(a-b)\left(\begin{smallmatrix}a&0\\0&b\end{smallmatrix}\right)-(a-b)^2C_t^2\bigr)(a-b)C_t\\
&=&\bigl(\bigl(\begin{smallmatrix}a-b-a^2+ab&0\\0&a-b-ba+b^2\end{smallmatrix}\bigr)C_t-
C_t(a-b)\left(\begin{smallmatrix}a&0\\0&b\end{smallmatrix}\right)-(a-b)^2\cos^2t\left(\begin{smallmatrix}1&0\\0&1\end{smallmatrix}\right)\bigr)(a-b)C_t\\
&=&\bigl(\bigl(\begin{smallmatrix}-b+ab&0\\0&a-ba\end{smallmatrix}\bigr)C_t-
C_t\left(\begin{smallmatrix}a-ba&0\\0&ab-b\end{smallmatrix}\right)-(a-b)^2\cos^2t\left(\begin{smallmatrix}1&0\\0&1\end{smallmatrix}\right)\bigr)(a-b)C_t\\
&=&\bigl(\bigl(\begin{smallmatrix}(ab+ba-a-b)\cos^2t&0\\0&(ab+ba-a-b)\cos^2t\end{smallmatrix}\bigr)-\bigl(\begin{smallmatrix}(a-b)^2\cos^2t&0\\0&(a-b)^2\cos^2t\end{smallmatrix}\bigr)\bigr)(a-b)C_t=0.
\end{eqnarray*}

Thus, the pair $(a\oplus b,b\oplus a)$ is homotopy equivalent to the pair $(a\oplus b,a\oplus b)$, and the latter is homotopy trivial by Lemma \ref{L1}.

\end{proof}

So we see that the equivalence classes of pairs satisfying the relations (\ref{main}) in matrix algebras over $A$ form an abelian group for any $C^*$-algebra $A$. Let us denote this group by $L(A)$.

Note that pairs of projections patently satisfy the relations (\ref{main}). If $A$ is a unital $C^*$-algebra then $K_0(A)$ consists of formal differences $[p]-[q]$ with $p,q$ projections in matrices over $A$. Then 
$$
\iota([p]-[q])=[(p,q)]
$$
gives rise to a morphism $\iota:K_0(A)\to L(A)$.

In the non-unital case, $\iota$ can be defined after unitalization. But, as we shall see later, unlike $K_0$, there is no need to unitalize for $L$. The following example shows the reason for that in the commutative case. 
\begin{example}
Let $X$ be a compact Hausdorff space, $x\in X$, $Y=X\setminus\{x\}$, $A=C_0(Y)$, $A^+=C(X)$. Let $[p]-[q]\in K_0(A)$, where $p,q\in M_n(A^+)$ are projections. Then $p=p_0+\alpha$, $q=p_0+\beta$, where $p_0$ is constant on $X$, and $\alpha,\beta\in M_n(A)$. Without loss of generality we may assume that $\alpha,\beta=0$ not only at the point $x$, but also in a small neighborhood $U$ of $x$. Let $h\in C(X)$ satisfy $0\leq h\leq 1$, $h(x)=0$ and $h(z)=1$ for any $z\in X\setminus U$. Set $a=hp_0+\alpha$, $b=hp_0+\beta$, then $a,b\in M_n(A)$ and $[(a,b)]\in L(A)$.

\end{example}  

\begin{lem}\label{point}
$L(\mathbb C)\cong\mathbb Z$.

\end{lem}
\begin{proof}
Let $a,b\in M_n$, $0\leq a,b\leq 1$. Let $e_1,\ldots,e_n$ (resp. $e'_1,\ldots,e'_n$) be an orthonormal basis of eigenvectors for $a$ (resp. for $b$) with eigenvalues $\lambda_1,\ldots,\lambda_n$ (resp. $\lambda'_1,\ldots,\lambda'_n$). Let $0<\lambda_i<1$. Then $e_i$ is an eigenvector for $a-a^2$ with a non-zero eigenvalue $\lambda_i-\lambda_i^2$. As $(a-a^2)(a-b)=0$, so $(a-b)(a-a^2)=0$, hence 
$$
(a-b)(a-a^2)(e_i)=(\lambda_i-\lambda_i^2)(a-b)(e_i)=0,
$$ 
thus $(a-b)(e_i)=0$, or, equivalently, $a(e_i)=b(e_i)$. As $e_i$ is an eigenvector for $a$, so it is an eigenvector for $b$ as well, $b(e_i)=\lambda_ie_i$. So, eigenvectors, corresponding to the eigenvalues $\neq 0,1$, are the same for $a$ and $b$.   

Re-order, if necessary, the eigenvalues so that 
$$
\lambda_1,\ldots,\lambda_k\in (0,1), \qquad \lambda_{k+1},\ldots,\lambda_n\in\{0,1\},
$$ 
and denote the linear span of $e_1,\ldots,e_k$ by $L$. Similarly, assume that 
$$
\lambda'_1,\ldots,\lambda'_{k'}\in (0,1),\qquad \lambda'_{k'+1},\ldots,\lambda'_n\in\{0,1\},
$$ 
and denote the linear span of $e'_1,\ldots,e'_{k'}$ by $L'$. As $e_1,\ldots,e_k\in L'$ and, symmetrically, $e'_1,\ldots,e'_{k'}\in L$, so $\dim L=\dim L'$, $k=k'$, and $\lambda_i=\lambda'_i$ for $i=1,\ldots,k$. 

Then $L^\perp$ is an invariant subspace for both $a$ and $b$, and the restrictions $a|_{L^\perp}$ and $b|_{L^\perp}$ are projections (as their eigenvalues equal 0 or 1).  We may write $a$ and $b$ as matrices with respect to the decomposition $L\oplus L^\perp$:
\begin{equation}\label{2}
a=\left(\begin{matrix}c&0\\0&p\end{matrix}\right);\qquad
b=\left(\begin{matrix}c&0\\0&q\end{matrix}\right),
\end{equation}
where $p,q$ are projections. The linear homotopy 
$$
a_t=\left(\begin{matrix}tc&0\\0&p\end{matrix}\right);\qquad
b_t=\left(\begin{matrix}tc&0\\0&q\end{matrix}\right),\qquad t\in[0,1],
$$
connects the pair $(a,b)$ with the pair $(p,q)+(0,0)$. Therefore, $L(\mathbb C)$ is a quotient of $\mathbb Z$ (which is the set of homotopy classes of pairs of projections modulo stable equivalence). To see that $L(\mathbb C)$ is exactly $\mathbb Z$, note that (\ref{2}) implies that $\tr(a-b)\in\mathbb Z$ for any $a,b$ satisfying the relations (\ref{main}), so this integer is homotopy invariant.

\end{proof}

\begin{remark}
One may think that the relations (\ref{main}) imply that $a,b$ are something like projections plus a common part and can be reduced to just a pair of projections by cutting out the common part. The following example shows that this is not that simple. 

\end{remark}

\begin{example}
Let $A=C(X)$, let $Y,Z$ be closed subsets in $X$ with $Y\cap Z=K$. Let $p,q\in M_n(C(Y))$ be projection-valued functions on $Y$ such that $p|_K=q|_K=r$, and let $r$ cannot be extended to a projection-valued function on $Z$ due to a $K$-theory obstruction, but can be extended to a matrix-valued function $s\in M_n(C(Z))$ on $Z$ (with $0\leq s\leq 1$). Then set $a=\left\lbrace\begin{matrix}p&{\rm on}&Y;\\s&{\rm on}&Z\end{matrix}\right.$ and $b=\left\lbrace\begin{matrix}q&{\rm on}&Y;\\s&{\rm on}&Z\end{matrix}\right.$. 

\end{example}

\section{Universal $C^*$-algebra for relations (\ref{main})}

Denote the $C^*$-algebra generated by $a,b$ satisfying (\ref{main}) by $C^*(a,b)$. The universal $C^*$-algebra is the least $C^*$-algebra $D$ such that for any $a,b$ with (\ref{main}) there is a surjective $*$-homomorphism $\varphi:D\to C^*(a,b)$, \cite{Loring-book}. `The least' means that for any surjective $*$-homomorphism $\psi:E\to C^*(a,b)$ there is a surjective $*$-homomorphism $\chi:E\to D$ such that $\psi=\varphi\circ\chi$.

Let $I\subset C^*(a,b)$ denote the ideal generated by $a-a^2$, and let $C^*(a,b)/I$ be the quotient $C^*$-algebra. Then $C^*(a,b)/I$ is generated by $\dot a=q(a)$ and $\dot b=q(b)$, where $q$ is the quotient map. But since $q(a-a^2)=q(b-b^2)=0$, $\dot a$ and $\dot b$ are projections, and $C^*(a,b)/I$ is generated by two projections. 

Then the $C^*$-algebra $C^*(a,b)$ is completely determined by the ideal $I$, by the quotient $C^*(a,b)/I$, and by the Busby invariant $\tau:C^*(a,b)/I\to Q(I)$ (we denote by $M(I)$ the multiplier algebra of $I$ and by $Q(I)=M(I)/I$ the outer multiplier algebra). The latter is defined by the two projections $\tau(\dot a),\tau(\dot b)\in Q(C_0(Y))$, where $X=\Sp(a)$, $Y=X\setminus\{0,1\}$. Let $C_b(Y)$ denote the $C^*$-algebra of bounded continuous functions on $Y$ and let 
$$
\pi:C_b(Y)\to C_b(Y)/C_0(Y)=Q(C_0(Y))
$$ 
be the quotient map. Using Gelfand duality, we identify $a$ with the function $\id$ on $\Sp(a)$. Let $f\in C_0(Y)$. Then 
$$
\tau(\dot a)\pi(f(a))=\tau(\dot b)\pi(f(a))=\pi(af(a)), 
$$
so we can easily calculate these two projections.

If $1\notin X$ then $\tau(\dot a)=\tau(\dot b)=0$; if $X=\{1\}$ then $I=0$; if $1\in X$ and $X$ has at least one more point $x$ then $\tau(\dot a)=\tau(\dot b)$ is the class of functions $f$ on $X$ such that $f(1)=1$ and $f(t)=0$ for all $t\leq x$.

Let $M_1\subset M_2$ denote the upper left corner in the 2-by-2 matrix algebra. Set 
$$
D=\{f\in C([-1,1];M_2): f(-1)=0, f(1) \ {\rm is \ diagonal}, f(t)\in M_1 \ \mbox{for}\ t\in(-1,0]\}.
$$
The structure of $D$ is similar to that of $C^*(a,b)$. The ideal 
$$
J=\{f\in D:f(t)=0 \ \mbox{for}\ t\in[0,1]\}\cong C_0(-1,0)
$$ 
is the universal $C^*$-algebra for $I$ (surjects on $I$ for any $0\leq a\leq 1$), and the quotient is the universal (nonunital) $C^*$-algebra 
\begin{equation}\label{qC}
D/J=\mathbb C\ast\mathbb C=\{m\in C([0,1],M_2): m(1) \mbox{\ is\ diagonal}, \, m(0)\in M_1\}
\end{equation}
generated by two projections \cite{Raeburn_Sinclair_two_projections}. Note that this $C^*$-algebra is an extension of $\mathbb C$ by the $C^*$-algebra $q\mathbb C=\{m\in C_0((0,1],M_2):m(1) \mbox{\ is\ diagonal}\}$ used in the Cuntz picture of $K$-theory. 

\begin{lem}
The $C^*$-algebra $D$ is universal for the relations (\ref{main}).

\end{lem}  
\begin{proof}
For any $a,b$ satisfying (\ref{main}) there are standard surjective $*$-homomorphisms $\alpha:J\to I$ and $\gamma:D/J\to C^*(a,b)/I$. Since $\alpha$ is surjective, it induces $*$-homomorphisms $M(\alpha):M(J)\to M(I)$ and $Q(\alpha):Q(J)\to Q(I)$ in a canonical way. As 
$$
D\cong\{(m,f):m\in M(J), f\in D/J, q_J(m)=\tau(f)\},
$$ 
$$
C^*(a,b)\cong\{(n,g):n\in M(I),g\in C^*(a,b)/I, q_I(n)=\sigma(g)\}, 
$$
where $q_\bullet:M(\bullet)\to Q(\bullet)$ is the quotient map, so the map $\beta:D\to C^*(a,b)$ can be defined by $\beta(m,f)=(M(\alpha)(m),\gamma(f))$. This map is well defined if the diagram
$$
\begin{xymatrix}{
D/J\ar[r]^-{\tau}\ar[d]^-{\gamma}&Q(J)\ar[d]^-{Q(\alpha)}\\
C^*(a,b)/I\ar[r]^-{\sigma}&Q(I)
}\end{xymatrix}
$$
commutes. It does commute. The case $X=\Sp(a)=\{1\}$ is trivial. For other cases, notice that the image of $\tau$ lies in $C_0(0,1]/C_0(0,1)\subset Q(J)$,
and the image of $\sigma$ lies in $C(X)/C_0(X\setminus\{0\})$, which is either $\mathbb C$ or 0 (when $1\in X$ or $1\notin X$ respectively), and the restriction of $Q(\alpha)$ from the image of $\tau$ to the image of $\sigma$ is induced by the inclusion $X\subset[0,1]$.

So, for any $A$ and any $a,b\in A$ satisfying (\ref{main}) there is a surjective $*$-homomorphism from $D$ to $C^*(a,b)$. To see that $D$ is universal it suffices to show that $D$ is generated by some $\mathbf{a},\mathbf{b}$ satisfying (\ref{main}). Set 
\begin{equation}\label{a}
\mathbf{a}(t)=\left\lbrace\begin{array}{cl}\left(\begin{matrix}\cos^2\frac{\pi}{2}t&0\\0&0\end{matrix}\right)&\mbox{for} \ t\in[-1,0];\\
\left(\begin{matrix}1&0\\0&0\end{matrix}\right)&\mbox{for} \ t\in[0,1],\end{array}\right.
\end{equation} 
\begin{equation}\label{b}
\mathbf{b}(t)=\left\lbrace\begin{array}{cl}\left(\begin{matrix}\cos^2\frac{\pi}{2}t&0\\0&0\end{matrix}\right)&\mbox{for} \ t\in[-1,0];\\
\left(\begin{matrix}\cos^2\frac{\pi}{2}t&\cos\frac{\pi}{2}t\sin\frac{\pi}{2}t\\\cos\frac{\pi}{2}t\sin\frac{\pi}{2}t&\sin^2\frac{\pi}{2}t\end{matrix}\right)&\mbox{for} \ t\in[0,1].\end{array}\right.
\end{equation}
Then $D$ is generated by these $\mathbf{a}$ and $\mathbf{b}$.

\end{proof}

The $C^*$-algebra $D$ allows one more description. Set $A_0=\mathbb C^2$, $F=\mathbb C\oplus M_2$ and define a $*$-homomorphism $\gamma:A_0\to F\oplus F$ by $\gamma=\gamma_0\oplus\gamma_1$, where $\gamma_0,\gamma_1:\mathbb C^2\to \mathbb C\oplus M_2$ are given by
$$
\gamma_0(\lambda,\mu)=\lambda\oplus\left(\begin{matrix}\lambda&0\\0&0\end{matrix}\right);\quad
\gamma_1(\lambda,\mu)=0\oplus\left(\begin{matrix}\lambda&0\\0&\mu\end{matrix}\right);\quad\lambda,\mu\in\mathbb C.
$$ 
Let $\partial:C([0,1];F)\to F\oplus F$ be the boundary map, $\partial(f)=f(0)\oplus f(1)$, $f\in C([0,1];F)$.
Then $D$ can be identified with the pullback
$$
\begin{xymatrix}{
D=A_1\ar[r]\ar[d]&A_0\ar[d]^-{\gamma}\\
C([0,1];F)\ar[r]^-{\partial}&F\oplus F,
}\end{xymatrix}
$$ 
$$
D=\{(f,a):f\in C([0,1];F), a\in A_0, \partial(f)=\gamma(a)\}.
$$
Such pullback is called a 1-dimensional noncommutative CW complex (NCCW complex) in \cite{ELP_Crelle}; in this terminology, $A_0$ is a 0-dimensional NCCW complex. 

Recall (\cite{Blackadar1985}) that a $C^*$-algebra $B$ is {\it semiprojective} if, for any $C^*$-algebra $A$ and increasing chain of ideals $I_n\subset A$, $n\in\mathbb N$, with $I=\overline{\cup_n I_n}$ and for any $*$-homomorphism $\varphi:B\to A/I$ there exists $n$ and $\hat{\varphi}:B\to A/I_n$ such that $\varphi=q\circ\hat{\varphi}$, where $q:A/I_n\to A/I$ is the quotient map.

\begin{cor}
The $C^*$-algebra $D$ is semiprojective.

\end{cor} 
\begin{proof}
Essentially, this is Theorem 6.2.2 of \cite{ELP_Crelle}, where it is proved that all unital 1-dimensional NCCW complexes are semiprojective. The non-unital case is dealt in Theorem 3.15 of \cite{Thiel-Diplomarbeit}, where is it noted that if $A_1$ is a 1-dimensional NCCW complex then $A_1^+$ is a 1-dimensional NCCW as well, and semiprojectivity of $A_1$ is equivalent to semiprojectivity of $A_1^+$.  

\end{proof}

One more picture of $D$ can be given in terms of amalgamated free product: $D=C(0,1]\ast_{C_0(0,1)}C(0,1]$.

\section{Identifying $L$ with $K_0$}

Our definition of $L(A)$ can be reformulated in terms of the universal $C^*$-algebra $D$ as
$$
L(A)=\varinjlim[D,M_n(A)],
$$
where $[-,-]$ denotes the set of homotopy classes of $*$-homomorphisms. Recall that semiprojectivity is equivalent to stability of relations that determine $D$, (Theorem 14.1.4 of \cite{Loring-book}). The latter means that for any $\varepsilon>0$ there exists $\delta>0$ such that whenever $c,d\in A$ satisfy 
$$
\|c\|\leq 1,\quad\|d\|\leq 1, \quad c,d\geq 0,\quad \|(c-c^2)(c-d)\|<\delta,\quad \|(d-d^2)(c-d)\|<\delta,
$$
there exist $a,b\in A$ such that $\|a-c\|<\varepsilon$, $\|b-d\|<\varepsilon$, and $a$, $b$ satisfy the relations (\ref{main}). 
Stability of the relations (\ref{main}) implies that
$$
L(A)=[D,A\otimes\mathbb K]=[[D,A\otimes\mathbb K]],
$$
where $\mathbb K$ denotes the $C^*$-algebra of compact operators, and $[[\cdot,\cdot]]$ is the set of homotopy classes of asymptotic homomorphisms.

\begin{lem}
The functor $L$ is half-exact.

\end{lem}
\begin{proof}
Let 
$$
\begin{xymatrix}{
0\ar[r]&I\ar[r]^-{i}&B\ar[r]^-{p}&A\ar[r]&0
}\end{xymatrix}
$$
be a short exact sequence of $C^*$-algebras. It is obvious that $p_*\circ i_*=0$, so it remains to check that $\Ker p_*\subset \Im i_*$.
Suppose that $a,b\in M_n(B)$ satisfy (\ref{main}) and $(p(a),p(b))=0$ in $L(A)$. This means that there is a homotopy connecting $(p(a),p(b))$ to $(0,0)$ in $M_k(A)$ for some $k\geq n$ such that the whole path satisfies (\ref{main}). This homotopy is given by a $*$-homomorphism $\psi:D\to C([0,1],M_k(A))$ such that $\ev_1\circ\psi=0$, where $\ev_t$ denotes the evaluation map at $t\in[0,1]$.

When $D$ is a semiprojective $C^*$-algebra, the homotopy lifting theorem  (\cite{Blackadar2012}, Theorem 5.1) asserts that, given a commuting diagram 
$$
\begin{xymatrix}{
D\ar@{-->}[drr]_-{\varphi}\ar[ddrr]_-{\psi}\ar[drrrr]^-{\kappa}&&&&\\
&&C([0,1];M_k(B))\ar[rr]_-{\ev_0}\ar[d]^-{\overline{p}_k}&&M_k(B)\ar[d]^-{p_k}\\
&&C([0,1];M_k(A))\ar[rr]^-{\ev_0}&&M_k(A),
}\end{xymatrix}
$$
where $\overline{p}_k$ and $p_k$ are the $*$-homomorphisms induced by a surjection $p$,
there exists a $*$-homomorphism $\varphi$ completing the diagram. Replacing $A$ and $B$ by matrices over these $C^*$-algebras, we get a lifting $\varphi$ for the given homotopy. As $\ev_1\circ\psi=0$, so $\ev_1\circ\varphi$ maps $D$ to $M_k(I)$. Thus the pair $(a,b)$ lies in the image of $i_*$.

\end{proof}

In a standard way, set $L_n(A)=L(S^nA)$, where $SA$ denotes the suspension over $A$. Then, by Theorem 21.4.3 of \cite{Blackadar-book}, $L_n(A)$, being homotopy invariant and half-exact, is a homology theory. Also, by Theorem 22.3.6 of \cite{Blackadar-book} and by Lemma \ref{point}, it coinsides with the $K$-theory on the bootstrap category of $C^*$-algebras. We shall show now that it coinsides with the $K$-theory for any $C^*$-algebra.

Set

%$$
%c(t)=\left(\begin{matrix}\left(\begin{smallmatrix}\sin^2\frac{\pi}{2}t&0\\0&1\end{smallmatrix}\right)&\left(\begin{smallmatrix}\cos\fr%ac{\pi}{2}t\sin\frac{\pi}{2}t&0\\
%0&0\end{smallmatrix}\right)\\
%\left(\begin{smallmatrix}\cos\frac{\pi}{2}t\sin\frac{\pi}{2}t&0\\0&0\end{smallmatrix}\right)&\left(\begin{smallmatrix}\cos^2\frac{\pi}%{2}t&0\\
%0&0\end{smallmatrix}\right)\end{matrix}\right), \qquad t\in[-1,0];
%$$

%$$
%d(t)=\left(\begin{matrix}\left(\begin{smallmatrix}0&0\\0&1\end{smallmatrix}\right)&\left(\begin{smallmatrix}0&0\\
%0&0\end{smallmatrix}\right)\\
%\left(\begin{smallmatrix}0&0\\0&0\end{smallmatrix}\right)&\left(\begin{smallmatrix}\cos^2\frac{\pi}{2}t&\cos\frac{\pi}{2}t\sin\frac{\p%i}{2}t\\
%\cos\frac{\pi}{2}t\sin\frac{\pi}{2}t&\sin^2\frac{\pi}{2}t\end{smallmatrix}\right)\end{matrix}\right), \qquad t\in[0,1];
%$$

%$$
%e(t)=\left(\begin{matrix}\left(\begin{smallmatrix}0&0\\0&1\end{smallmatrix}\right)&\left(\begin{smallmatrix}0&0\\
%0&0\end{smallmatrix}\right)\\
%\left(\begin{smallmatrix}0&0\\0&0\end{smallmatrix}\right)&\left(\begin{smallmatrix}1&0\\
%0&0\end{smallmatrix}\right)\end{matrix}\right), \qquad t\in[0,1];
%$$

%$$
%P(t)=\left\lbrace\begin{array}{cl}c(t)&\mbox{for\ }t\in[-1,0],\\
%e(t)&\mbox{for\ }t\in[0,1]\end{array}\right. ,\qquad 
%Q(t)=\left\lbrace\begin{array}{cl}c(t)&\mbox{for\ }t\in[-1,0],\\
%d(t)&\mbox{for\ }t\in[0,1]\end{array}\right. .
%$$

$$
P=\left(\begin{matrix}1-\mathbf{b}&f(\mathbf{a})\\f(\mathbf{a})&\mathbf{a}\end{matrix}\right);\quad
Q=\left(\begin{matrix}1-\mathbf{b}&f(\mathbf{a})\\f(\mathbf{a})&\mathbf{b}\end{matrix}\right),
$$
where $\mathbf{a}$, $\mathbf{b}$ are generators for $D$ ((\ref{a}),(\ref{b})), and $f\in C_0(0,1)$ is given by $f(t)=(t-t^2)^{1/2}$. Then $P,Q\in M_2(D^+)$, where $D^+$ denotes the unitalization of $D$.

By Lemma \ref{ideal1}, $f(\mathbf{a})=f(\mathbf{b})$ and $\mathbf{a}f(\mathbf{a})=\mathbf{b}f(\mathbf{a})$, so $P$ and $Q$ are projections. One also has $P-Q\in M_2(D)$, hence 
$$
x=[P]-[Q]\in K_0(D).
$$

\begin{lem}
$K_0(D)\cong\mathbb Z$ with $x$ as a generator.

\end{lem}
\begin{proof}
Consider the short exact sequence
$$
\begin{xymatrix}{
0\ar[r]&J\ar[r]&D\ar[r]^-{\pi}&\mathbb C\ast\mathbb C\ar[r]&0,
}\end{xymatrix}
$$
where $\mathbb C\ast\mathbb C$ is the universal (nonunital) $C^*$-algebra (\ref{qC}) generated by two projections, $p$ and $q$ \cite{Raeburn_Sinclair_two_projections}, and $\pi$ is given by restriction to $[0,1]$, $\pi(\mathbf{a})=p$, $\pi(\mathbf{b})=q$. We have $\pi(P)=(1-q)\oplus p$, $\pi(Q)=(1-q)\oplus q$, so $\pi_*(x)=[p]-[q]\in K_0(\mathbb C\ast\mathbb C)$. As $P(t)=Q(t)$ when $t\in[-1,0]$, so for the boundary (exponential) map $\delta:K_0(\mathbb C\ast\mathbb C)\to K_1(J)$ we have $\delta(P)=\delta(Q)$. Recall that $J\cong C_0(-1,0)$. Direct calculation shows that $\delta(P)=\delta(Q)\neq 0$. The claim follows now from the $K$-theory exact sequence
$$
\begin{xymatrix}{
0=K_0(J)\ar[r]&K_0(D)\ar[r]^-{\pi_*}&K_0(\mathbb C\ast\mathbb C)\ar[r]^-{\delta}&K_1(J)\cong\mathbb Z.
}\end{xymatrix}
$$

\end{proof}

Let us define a map $\kappa:L(A)\to K_0(A)$. If $l=[(a,b)]\in L(A)$ then the pair $(a,b)$ determines a $*$-homomorphism $\varphi:D\to M_n(A)$ by $\varphi(\mathbf{a})=a$; $\varphi(\mathbf{b})=b$. So, $l\in L(A)$ determines a $*$-homomorphism $\varphi$ up to homotopy (for some $n$). Put 
$$
\kappa(l)=\varphi_*(x)\in K_0(A).
$$ 
As this definition is homotopy invariant and as direct sum of pairs corresponds to direct sum of $*$-homomorphisms, so the map $\kappa$ is a well defined group homomorphism.

Recall that there is also a map $\iota:K_0(A)\to L(A)$ given by $\iota([p]-[q])=[(p,q)]$, where $[p]-[q]\in K_0(A)$.

\begin{lem}
For any unital $C^*$-algebra $A$, one has $\kappa\circ\iota=\id_{K_0(A)}$; $\iota\circ\kappa=\id_{L(A)}$, hence $L(A)=K_0(A)$.

\end{lem} 
\begin{proof}

To show the first identity, let $z\in K_0(A)$ and let $p,q\in M_n(A)$ be projections such that $z=[p]-[q]$. Let $\varphi:D\to M_n(A)$ be a $*$-homomorphism determined by the pair $(p,q)$. Then, due to universality of $\mathbb C\ast\mathbb C$, $\varphi$ factorizes through $\mathbb C\ast\mathbb C$, $\varphi=\psi\circ\pi$, where $\pi:D\to \mathbb C\ast\mathbb C$ is the quotient map and $\psi:\mathbb C\ast\mathbb C\to M_n(A)$ is determined by $\psi(i_1(1))=p$ and $\psi(i_2(1))=q$, where $i_1,i_2:\mathbb C\to \mathbb C\ast\mathbb C$ are inclusions onto the first and the second copy of $\mathbb C$. Then 
$$
\varphi(x)=\psi_*([i_1(1)]-[i_2(1)])=[p]-[q], 
$$
hence $\kappa(\iota(z))=z$.

Let us show the second identity. For $[(a,b)]\in L(A)$, let $\varphi:D\to M_n(A)$ be a $*$-homomorphism defined by the pair $(a,b)$ (i.e. by $\varphi(\mathbf{a})=a$, $\varphi(\mathbf{b})=b$), and let $\varphi^+:D^+\to M_n(A)$ be its extension, $\varphi^+(1)=1$. Then $\iota(\kappa([(a,b)]))=[(\varphi^+_2(P),\varphi^+_2(Q))]$, where $\varphi^+_2=\varphi^+\otimes\id_{M_2}$. 

For $s\in[0,1]$, set
%$$
%c_s(t)=\left(\begin{matrix}\left(\begin{smallmatrix}s\cdot\sin^2\frac{\pi}{2}t&0\\0&s\cdot %1\end{smallmatrix}\right)&\left(\begin{smallmatrix}s\cdot\cos\frac{\pi}{2}t\sin\frac{\pi}{2}t&0\\
%0&0\end{smallmatrix}\right)\\
%\left(\begin{smallmatrix}s\cdot\cos\frac{\pi}{2}t\sin\frac{\pi}{2}t&0\\0&0\end{smallmatrix}\right)&\left(\begin{smallmatrix}\cos^2\fra%c{\pi}{2}t&0\\
%0&0\end{smallmatrix}\right)\end{matrix}\right), \qquad t\in[-1,0];
%$$

%$$
%d_s(t)=\left(\begin{matrix}\left(\begin{smallmatrix}0&0\\0&s\cdot 1\end{smallmatrix}\right)&\left(\begin{smallmatrix}0&0\\
%0&0\end{smallmatrix}\right)\\
%\left(\begin{smallmatrix}0&0\\0&0\end{smallmatrix}\right)&\left(\begin{smallmatrix}\cos^2\frac{\pi}{2}t&\cos\frac{\pi}{2}t\sin\frac{\p%i}{2}t\\
%\cos\frac{\pi}{2}t\sin\frac{\pi}{2}t&\sin^2\frac{\pi}{2}t\end{smallmatrix}\right)\end{matrix}\right), \qquad t\in[0,1];
%$$

%$$
%e_s(t)=\left(\begin{matrix}\left(\begin{smallmatrix}0&0\\0&s\cdot 1\end{smallmatrix}\right)&\left(\begin{smallmatrix}0&0\\
%0&0\end{smallmatrix}\right)\\
%\left(\begin{smallmatrix}0&0\\0&0\end{smallmatrix}\right)&\left(\begin{smallmatrix}1&0\\
%0&0\end{smallmatrix}\right)\end{matrix}\right), \qquad t\in[0,1];
%$$

%$$
%P_s(t)=\left\lbrace\begin{array}{cl}c_s(t)&\mbox{for\ }t\in[-1,0],\\
%e_s(t)&\mbox{for\ }t\in[0,1]\end{array}\right. ,\qquad 
%Q_s(t)=\left\lbrace\begin{array}{cl}c_s(t)&\mbox{for\ }t\in[-1,0],\\
%d_s(t)&\mbox{for\ }t\in[0,1]\end{array}\right. .
%$$

$$
P_s=C_sPC_s;\quad Q_s=C_sQC_s,\qquad\mbox{where\ }C_s=\left(\begin{matrix}s\cdot 1&0\\0&1\end{matrix}\right).
$$

Then 
$$
P_s,Q_s\in M_2(D^+),\quad P_s-Q_s\in M_2(D),\quad 0\leq P_s,Q_s\leq 1,
$$ 
$$
(P_s-P^2_s)(P_s-Q_s)=0,\quad (Q_s-Q_s^2)(P_s-Q_s)=0
$$ 
for all $s\in [0,1]$; $P_0,Q_0\in M_2(D)$, and 
$$
P_1=P,\quad Q_1=Q;\qquad P_0=\left(\begin{matrix}0&0\\0&\mathbf{a}\end{matrix}\right),\quad Q_0=\left(\begin{matrix}0&0\\0&\mathbf{b}\end{matrix}\right). 
$$
Therefore, $(\varphi_2^+(P_s),\varphi_2^+(Q_s))$ provides a homotopy connecting $(\varphi_2^+(P),\varphi_2^+(Q))$ with $(0\oplus a,0\oplus b)$, hence, the pair $(\varphi_2^+(P),\varphi_2^+(Q))$ is equivalent to the pair $(a,b)$.

\end{proof}

\begin{thm}
The functors $L$ and $K_0$ coinside for any $C^*$-algebra $A$.

\end{thm}
\begin{proof}
Both functors are half-exact and coinside for unital $C^*$-algebras, so the claim follows. 

\end{proof}

\begin{remark}
Similarly to $D$, one can define a $C^*$-algebra $D_B$ for any $C^*$-algebra $B$ as an appropriate extension of $B\ast B$ by $CB$, where $CB$ is the cone over $B$ (or by $D_B=CB\ast_{SB}CB$). Then one gets the group $[D_B,A\otimes\mathbb K]$. Regretfully, $D_B$ has no nice presentation (unlike $D=D_\mathbb C$), so we don't pursue here the bivariant version. 

\end{remark}

\end{document}